\documentclass[english,11pt]{article}
\usepackage[margin=2.5cm]{geometry}
\usepackage[hyphens]{url}
\usepackage[hidelinks]{hyperref}
\usepackage{amsthm,amsfonts,amssymb,amsmath}
\usepackage{graphicx}
\usepackage{cleveref}
\usepackage[square,sort,comma,numbers]{natbib}

\newtheorem{theorem}{Theorem}[section]

\Crefname{prop}{Proposition}{Propositions}
\newtheorem{lemma}[theorem]{Lemma}

\newtheorem{conj}[theorem]{Conjecture}
\newtheorem{qn}[theorem]{Question}

\newtheorem{definition}[theorem]{Definition}

\theoremstyle{definition}

\numberwithin{equation}{section}
\numberwithin{figure}{section}
\allowdisplaybreaks

\usepackage{cleveref}
\usepackage{babel}
\usepackage[T1]{fontenc}
\usepackage[utf8]{inputenc}

\newcommand{\eps}{\varepsilon}

\newcommand\car{\mathbin{\text{\scalebox{.84}{$\square$}}}}

\date{} 
\title{Long cycles in vertex transitive digraphs}

\author{Matija Buci\'c\thanks{Faculty of Mathematics, University of Vienna, Austria, and Department of Mathematics, Princeton University, USA. Email: \texttt{matija.bucic@univie.ac.at}. Research supported in part by NSF Award DMS-2349013.} \and Kevin Hendrey\thanks{School of Mathematics, Monash University, Australia. Research supported by the Australian Research Council.} \and Bojan Mohar\thanks{Department of Mathematics, Simon Fraser University, Burnaby, BC, Canada. 
Email: \texttt{mohar@sfu.ca}. Research supported in part by the NSERC Discovery Grant R832714 (Canada), by the ERC Synergy grant (European Union, ERC, KARST, project number 101071836), and by the Research Core Grant P1-0297 of ARIS (Slovenia). On leave from: FMF, Department of Mathematics, University of Ljubljana, Ljubljana, Slovenia.} \and Raphael Steiner\thanks{Department of Mathematics, ETH Z\"{u}rich, Switzerland. Email: \texttt{raphaelmario.steiner@math.ethz.ch}. Research supported by SNSF Ambizione grant.} \and Liana Yepremyan\thanks{Department of Mathematics, Emory University, Atlanta, GA 30322. Email: {\tt lyeprem@emory.edu}. Research supported by NSF Award DMS-2247013.}}

\usepackage{xcolor}

\begin{document}
\maketitle
\begin{abstract}
One of the most well-known conjectures concerning Hamiltonicity in graphs asserts that any sufficiently large connected vertex transitive graph contains a Hamilton cycle. In this form, it was first written down by Thomassen in 1978, inspired by a closely related conjecture due to Lov\'asz from 1969. It has been attributed to several other authors in a survey on the topic by Witte and Gallian in 1984. 

The analogous question for vertex transitive digraphs has an even longer history, having been first considered by Rankin in 1946. It is arguably more natural from the group-theoretic perspective underlying this problem in both settings. Trotter and Erd\H{o}s proved in 1978 that there are infinitely many connected vertex transitive digraphs which are not Hamiltonian. This left open the very natural question of how long a directed cycle one can guarantee in a connected vertex transitive digraph on $n$ vertices. 

In 1981, Alspach asked if the maximum perimeter gap
(the gap between the circumference and the order of the digraph)
is a growing function in $n$. 
We answer this question in the affirmative, showing that it grows at least as fast as $(1-o(1)) \ln n$. 
On the other hand, we prove that one can always find a directed cycle of length at least $\Omega(n^{1/3})$, establishing the first lower bound growing with $n$, providing a directed analogue of a famous result of Babai from 1979 in the undirected setting.
\end{abstract} 

\section{Introduction}

Finding long paths and cycles in graphs is one of the most classical directions of study in graph theory. Perhaps the most famous instance of this general direction is the question of finding the longest possible cycle, namely one that traverses all the vertices. Such a cycle is called a \emph{Hamilton cycle}, and a graph containing it is said to be \emph{Hamiltonian}. Hamiltonicity is a very classical and extensively studied graph property. In general, it is a hard problem to decide if a given graph is Hamiltonian. In fact, this is one of Karp's famous 21 NP-hard problems \cite{karp}, and is often used to establish hardness of other computational problems.

This goes a long way towards explaining why there are so many interesting results establishing sufficient conditions for Hamiltonicity. The simplest one, featured in essentially every introductory course on graph theory, is Dirac's theorem from 1952, which states that any graph with minimum degree at least $\frac{n}{2}$ is Hamiltonian. 
A major downside of Dirac's theorem is that it only applies for very dense graphs, leading to a more challenging question of finding natural graph properties that would force Hamiltonicity even for much sparser graphs. Perhaps the most intriguing
candidate, first considered in the 1960s, is symmetry. 
This idea first appeared in a communication by Lov\'asz \cite{lovasz1969} from 1969, where he conjectured that any connected vertex transitive graph must contain a Hamilton path. Thomassen (see \cite{babai}) refined this conjecture in 1978 by asserting that any sufficiently large connected vertex transitive graph is Hamiltonian. 

These conjectures have attracted an immense amount of work over the years, with multiple surveys on the topic \cite{cayley-survey-84,pak-radoicic-survey,alspach81,curran-gallian-survey} being written, starting as early as the 1980s. Despite this attention, both conjectures remain widely open and most of what is known concerns various additional assumptions under which the conjectures hold. On the other hand, Babai \cite{babai} already in 1979 initiated a very general direction of attack, namely of trying to at least find a long cycle in \emph{every} connected vertex transitive graph (without any additional assumptions). In particular, he proved that an $n$-vertex connected vertex transitive graph always has a cycle of length $\Omega(\sqrt{n})$. In recent years, there have been several improvements over this result, all making interesting connections to a number of other interesting graph theoretic problems \cite{yepremyan,norin2025small,devos,ma2025intersections} leading to the current state of the art of $\Omega(n^{9/14})$ proved in \cite{norin2025small}. 

In this paper, we are interested in the directed analogue of this problem. Namely, how long a cycle can we find in any $n$-vertex connected vertex transitive digraph?  This question is even more classical, owing to the fact that in the arguably most interesting instance of the problem, namely that of Cayley digraphs, the directed variant is more natural\footnote{The key difference is that in the digraph case, one 
deals with arbitrary generating sets, not just symmetric ones, and this is more natural from the group theoretic perspective.} 
and translates to a natural group rearrangement problem. Indeed, the two oldest papers to consider this problem started from this group theoretic question and translated it to the Cayley digraph instance of our problem. These were a 1946 paper by Rankin \cite{rankin}, and an independent 1959 paper by Rapaport-Strasser \cite{RS}. Both of these works attribute their motivation to Campanology (as well as the ``knight tour'' problem in the latter case), where various instances of this problem have been solved in practice, more than a century before (see \cite[Section 4]{rankin} or \cite[Chapter 15]{white1985graphs} for more details).

The directed analog of Thomassen's conjecture was first disproved\footnote{Technically, the argument from \cite{TrotterErdos} relies on the existence of infinitely many so-called Sophie-Germain primes, which is still an open problem in number theory to this day. We discuss this in more detail in \Cref{sec:gap}.} by Trotter and Erd\H{o}s \cite{TrotterErdos} in 1978, who exhibited an infinite family of connected vertex transitive digraphs without Hamilton cycles. Motivated by this result, Alspach asked already in 1981 whether such graphs need to be at least ``nearly'' Hamiltonian. Here, to formalize this question, we can use the \emph{perimeter gap}, defined as the difference between the number of vertices and the length of the longest directed cycle in a digraph, as a measure of how far from Hamiltonian a graph is. In particular, Alspach (Question~7 in~\cite{alspach81}) asked if for any constant $C$ there exists a connected vertex transitive digraph with perimeter gap larger than $C$. We answer this question in the affirmative in the following quantitative form.

\begin{theorem}\label{thm:cycles-short}
    For infinitely many natural numbers $n$, there exists a connected vertex transitive digraph on $n$ vertices with perimeter gap at least $(1-o(1))\ln(n)$. 
\end{theorem}

Given this result, the question of how long of a cycle we can actually guarantee, raised by Babai in the undirected case already in 1979, becomes even more natural in the directed case. Here, it was not even known whether one can guarantee a cycle of length growing with $n$. We prove such a result, establishing a directed analog of Babai's result from 1979.

\begin{theorem}\label{thm:cycles-long}
    In any connected vertex transitive digraph $D$ on $n\ge 2$ vertices there is a directed cycle of length at least $\Omega(n^{1/3})$.
\end{theorem}

We note that neither Babai's result nor the aforementioned recent results improving on it extend to the directed case. This is because all of these proofs at various places crucially rely on the fact that in a $2$-connected graph $G$, any two longest cycles must intersect. However, an analogous statement fails badly for strongly $2$-connected digraphs, even if they are assumed to be $2$-out- and $2$-in-regular: In Figure~\ref{fig:disjoint} we illustrate a construction of strongly $2$-connected $2$-regular digraphs with arbitrarily many vertex-disjoint longest directed cycles. 

\begin{figure}[h]
    \centering
\includegraphics[width=0.9\linewidth]{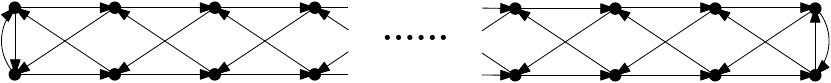}
    \caption{A family of strongly $2$-connected $2$-regular digraphs such that any longest directed cycle has length four, while the maximum length of a directed path is not bounded. In this family, the number of vertex-disjoint longest directed cycles can be arbitrarily large.}
    \label{fig:disjoint}
\end{figure}

In fact, our arguments provide slightly longer directed paths in place of directed cycles, namely of length at least $\Omega(\sqrt{n})$, see \Cref{thm:paths-long}. Using this, we also get a new proof of Babai's result in the undirected setting where the longest cycle and path are always within a constant factor of each other, see \cite{path-vs-cycle}. However, the same family of examples illustrated in Figure~\ref{fig:disjoint} shows that in general it does not suffice to use the existence of long directed paths and strong $2$-connectivity to infer the existence of long directed cycles. Hence, to prove Theorem~\ref{thm:cycles-long}, we need additional ideas, beyond the comparatively simpler ideas that suffice to guarantee a long directed path.

\textbf{Notation.} In this paper, all graphs and digraphs are finite, unless explicitly mentioned otherwise. Given a digraph $D$ we denote by $V(D)$ its vertex set, and by $A(D)$ the set of its arcs. Given $U \subseteq V(D)$ we denote by $N^{+}(U)$ the external out-neighborhood of $U$, namely the set of vertices not in $U$ which are endpoints of arcs with the start vertices in $U$. $N^{-}(U)$ is similarly defined as the external in-neighborhood. Given $u,v \in V(D)$ we write $d(u,v)$ for the directed distance from $u$ to $v$, namely the length of a shortest directed path from $u$ to $v$ (setting it to be infinity if such a path does not exist). The directed diameter of $D$ is the maximum of $d(u,v)$ over all $u,v \in V(D).$ A digraph is said to be strongly connected if its diameter is finite. It is said to be connected if its underlying graph, obtained by removing the directions from all edges, is connected. It is not hard to see that if $D$ is a vertex transitive digraph then these two notions coincide. Given an undirected graph $G$ and vertices $u,v\in V(G)$, we denote by $d(u,v)=d(v,u)$ the distance between $u$ and $v$ in $G$. The diameter of $G$ is $\sup_{u,v\in V(G)}d(u,v)$.

\section{Finding long paths and cycles in vertex transitive digraphs}

In this section, we prove \Cref{thm:cycles-long}. We begin here with a brief sketch and overview of the proof. 

The proof naturally splits into two parts depending on how large the directed diameter of our graph is. We deal with the case when the diameter is small in \Cref{subs:expansion}. In this case, we will be able to show that the graph must have some weak (sublinear) expansion properties (see \Cref{lem:transitive-implies-expander}) which combined with a directed variant of a DFS-type argument for finding long paths or cycles in expanders, gives us a long cycle (see \Cref{lem:long-cycles-in-expanders}). We deal with the case of large directed diameter in \Cref{subs:cycle-graph}. In this case, a key idea is to reduce our directed problem to an undirected one. To do so, we introduce an auxiliary undirected graph (the cycle graph, see \Cref{defn:cycle-graph}). We show that (unless our digraph already has a long directed cycle) our auxiliary graph inherits the large diameter property (see \Cref{lem:diameter-in-cycle-graph}). It also inherits a slightly weaker version of vertex transitivity (which we dub ``near transitivity'', see \Cref{defn:near-transitive}). While a long cycle in our auxiliary undirected graph might not on its own translate back to a long directed cycle in the original digraph, we show that a long \emph{induced} cycle does suffice (see \Cref{lem:induced-to-directed-cycles}). Finally, we show that in a nearly transitive graph with a large diameter, one can, in fact, find a long induced cycle (see \Cref{lem:long-cycle-in-cycle-graphs}), which completes the proof.

\subsection{Expansion properties of vertex transitive digraphs}\label{subs:expansion}

In this subsection, we will deal with the small diameter case of \Cref{thm:cycles-long}. We start by introducing the notion of graph expansion that we will work with. 

\begin{definition}\label{defn:expansion}
A digraph $D$ is an \emph{$\alpha$-expander} if $|V(D)|\geq 2$ and for every $U \subseteq V(D)$ with $|U| \le \frac23|V(D)|$, we have $|N^{+}(U)| \ge \alpha |U|$ and $|N^{-}(U)| \ge \alpha |U|$.
\end{definition}

The following lemma shows that vertex transitive digraphs with small diameter must have some weak expansion properties.

\begin{lemma}\label{lem:transitive-implies-expander}
  Let $D$ be a connected, vertex transitive digraph with directed diameter $d$. Then, $D$ is a $\frac{1}{3d}$-expander.
\end{lemma}

To prove this lemma, we will extend the proof of a classical result due to Babai on the vertex-expansion of vertex transitive graphs~\cite{Babai91} to the setting of vertex transitive digraphs. As in Babai's proof, it will be convenient to first consider the special case of Cayley digraphs and then reduce the general case to it. Recall that given a finite group $(\Gamma,\cdot)$ and a set of generators $S$ of $\Gamma$, the \emph{Cayley digraph} $\mathrm{Cay}(\Gamma,S)$ has vertex-set $\Gamma$ and an arc $(x,y)$ from a group element $x$ to another group element $y$ if and only if $x^{-1}y\in S$. The following statement implies that Cayley digraphs are $\frac{1}{3d}$-expanders, where $d$ denotes the directed diameter of the Cayley digraph.

\begin{lemma}\label{lem:cayleydigraph}
Let $\Gamma$ be a finite group, $d\in \mathbb{N}$ and suppose that $S\subseteq \Gamma$ is such that every element of $\Gamma$ can be written as a product of at most $d$ elements of $S$ (empty product and repetitions allowed). Then, for every subset $X\subseteq \Gamma$ with $|X|\le \frac{2}{3}|\Gamma|$ there exists some $s\in S$ such that $|Xs\setminus X|\ge \frac{|X|}{3d}$. 
\end{lemma}

\begin{proof}
We start by observing that there exists some $g\in \Gamma$ such that $|Xg\setminus X|\ge \frac{1}{3}|X|$. 
To see this, note that if we sample $g\in \Gamma$ uniformly at random, then for every $h\in \Gamma$ the probability that $h\in Xg$ equals $\frac{|X|}{|\Gamma|}$, and hence the expected size of $|Xg\setminus X|$ is exactly $\frac{|X|}{|\Gamma|}(|\Gamma|-|X|)\ge \frac{1}{3}|X|$. Thus, such a $g\in\Gamma$ must exist.

Let now $s_1,\ldots,s_t\in S$ with $t\le d$ be such that $g=s_t\cdots s_1$. For $0\le i \le t$ define $g_i:=s_i\cdots s_1$. We then have 
$$Xg\setminus X\subseteq \bigcup_{i=1}^{t}(Xg_i\setminus Xg_{i-1}).$$
Hence, there exists some $i\in [t]$ such that $|Xg_i\setminus Xg_{i-1}|\ge \frac{1}{t}|Xg\setminus X|\ge \frac{1}{3d}|X|$. This shows that $|Xs_i\setminus X|=|(Xs_i\setminus X)g_{i-1}|=|Xg_i\setminus Xg_{i-1}|\ge \frac{1}{3d}|X|$. Since $s_i\in S$, this concludes the proof.
\end{proof}

The desired result for the expansion of vertex transitive digraphs is now an easy consequence of the above.

\begin{proof}[Proof of Lemma~\ref{lem:transitive-implies-expander}]
Fix an arbitrary vertex $v\in V(D)$, let $\Gamma$ denote the automorphism group of $D$, and let $S\subseteq \Gamma$ be defined as the set of all automorphisms $\phi\in \Gamma$ such that $\phi(v)$ is an out-neighbor of $v$ in $D$. We claim that every element of $\Gamma$ can be written as a composition of at most $d$ elements of $S$. Indeed, let $\phi\in \Gamma$ be given arbitrarily. Then, since $D$ has directed diameter $d$, there exists some $t\le d$ and a sequence of vertices $v=x_0,x_1,\ldots,x_t=\phi(v)$ such that $(x_{i-1},x_i)\in A(D)$ for all $i\in [t]$. For $i=0,\ldots,t$, let $\phi_i$ denote an automorphism such that $\phi_i(v)=x_i$, where we let $\phi_0$ be the identity and $\phi_t:=\phi$. Then we can see that 
$$\phi=\phi_0^{-1}\circ\phi_t=(\phi_0^{-1}\circ \phi_{1})\circ (\phi_{1}^{-1}\circ \phi_{2})\circ \cdots \circ (\phi_{t-1}^{-1}\circ \phi_t).$$
By construction each of the automorphisms $\phi_{i-1}^{-1}\circ \phi_{i}$ for $i\in [t]$ sends $v$ to an out-neighbor of $v$. Hence, we have written $\phi$ as a composition of at most $d$ members of $S$, as desired. 

Let $n:=|V(D)|$ and let $U\subseteq V(D)$ be an arbitrary subset with $|U|\le \frac{2}{3}n$. Let $X\subseteq \Gamma$ be defined as the set of all automorphisms $\phi\in \Gamma$ such that $\phi(v)\in U$. Observe that $|X|=\frac{|U|}{n}\cdot|\Gamma|\le \frac{2}{3}|\Gamma|$. We can now apply Lemma~\ref{lem:cayleydigraph} to $\Gamma$, the set $S$ and the set $X$ to find that there exists some $s\in S$ such that $|Xs\setminus X|\ge \frac{1}{3d}|X|$. In other words, there exist at least $\frac{1}{3d}|X|$ automorphisms $\phi$ of $D$ with $\phi(v)\in U$ and $\phi(s(v))\notin U$. By the definition of $S$, in this situation we then always have $\phi(s(v))\in N_D^+(U)$. Since $D$ is vertex transitive, for every vertex $x\in N_D^+(U)$ there are exactly $\frac{|\Gamma|}{n}$ automorphisms $\phi$ with $\phi(s(v))=x$. Hence, the above yields at least $\frac{\frac{1}{3d}|X|}{|\Gamma|/n}=\frac{\frac{1}{3d}\frac{|U|}{n}|\Gamma|}{|\Gamma|/n}=\frac{1}{3d}|U|$ distinct vertices in $N_D^+(U)$. This establishes the desired inequality $|N_D^+(U)|\ge \frac{1}{3d}|U|$, concluding the proof.
\end{proof}

The following lemma guarantees a long cycle in an expanding digraph. It is a directed variant of the so-called DFS lemma, see \cite{michael-survey}.

\begin{lemma}\label{lem:long-cycles-in-expanders}
  An $n$-vertex $\alpha$-expander $D$ with $\alpha>0$ contains a directed cycle of length at least $\frac{\alpha n}{3}$.
\end{lemma}
\begin{proof}
Let $v_1$ be an arbitrary vertex in $D$. Since $D$ is an $\alpha$-expander, $D$ is strongly connected and every vertex in $D$ is a descendant of $v_1$. As long as we can, we repeat the following process. Given a directed path $v_1v_2\ldots v_i,$ if there exists an outneighbor of $v_i$ that is not in $\{v_1,\ldots, v_{i-1}\}$ and has at least $\frac{2n}{3}$ descendants in the digraph induced on $V(D) \setminus \{v_1,\ldots, v_{i}\}$, we choose such an outneighbor $v_{i+1}$ and continue. This process must eventually stop since the size of the induced subdigraph we consider decreases by one in each step. Let us suppose this process stops after $t$ steps with a path $P=v_1v_2\ldots v_t$ with $v_t$ having at least $\frac{2n}{3}$ descendants in the subdigraph induced on $V(D) \setminus \{v_1,\ldots, v_{t-1}\}$ but each of its outneighbors not on $P$ has at most $\frac{2n}{3}$ descendants in the subdigraph induced on $V(D) \setminus \{v_1,\ldots, v_{t}\}$. This implies that there is a subset $S$ of outneighbors of $v_t$ in $D-V(P)$ such that the set $U$ of all descendants (in $D-V(P)$) of the vertices in $S$ has at least $\frac{n}{3}$ and at most $\frac{2n}{3}$ vertices. 
Since $D$ is an expander, $U$ has at least $\frac{\alpha n}{3}$ outneighbors in $D$. All these outneighbors must be in $P$, so taking the one that is closest to $v_1$ on $P$, gives us a cycle in $D$ of at least that length.
\end{proof}

As we shall see next, combining the above two lemmas already yields a short proof of the aforementioned variant of \Cref{thm:cycles-long} where we seek a path instead of a cycle.

\begin{theorem}\label{thm:paths-long}
    In any connected vertex transitive digraph on $n$ vertices, there is a directed path of length at least $\Omega(n^{1/2})$.
\end{theorem}

\begin{proof}
Let $D$ be any given connected vertex transitive digraph on $n$ vertices. One easily checks that then $D$ must also be regular (i.e., there exists some $r\in \mathbb{N}$ such that every vertex has out- and in-degree exactly $r$). Indeed, vertex-transitivity directly implies that there exist numbers $r_1, r_2$ such that every vertex of $D$ has out-degree $r_1$ and in-degree $r_2$. But then the number of arcs in $D$ equals $r_1n$ and $r_2n$, so we must have $r_1=r_2$, as desired. In particular, $D$ is an Eulerian digraph and hence strongly connected. Let $d\in \mathbb{N}$ denote the directed diameter of $D$. By Lemma~\ref{lem:transitive-implies-expander} we have that $D$ is a $\frac{1}{3d}$-expander, and by applying Theorem~\ref{lem:long-cycles-in-expanders} we find that $D$ contains a directed cycle (and hence path) of length at least $\frac{n}{9d}$. On the other hand, since $D$ is strongly connected with diameter $d$, there exists a directed path of length $d$ in $D$. All in all, it follows that the maximum length of a directed path in $D$ is at least 
$$\max\left\{\frac{n}{9d},d\right\}\ge \sqrt{\frac{n}{9d}\cdot d}=\frac{1}{3}\sqrt{n}.$$ This concludes the proof of the theorem.
\end{proof}

We note that we may apply this theorem in the undirected setting as well (by considering the vertex transitive digraph obtained by replacing every edge with two arcs joining its two vertices, one in each direction). Since in a vertex transitive graph the length of a longest path and cycle differs only by a constant factor \cite{path-vs-cycle}, this gives a new proof of the result of Babai from 1979, guaranteeing the existence of a cycle of length $\Omega(\sqrt{n})$ in any connected vertex transitive graph \cite{babai}. 

\subsection{The cycle graph}\label{subs:cycle-graph}

The following definition of an auxiliary graph is going to prove very useful in finding long cycles.

\begin{definition}\label{defn:cycle-graph}
Given a digraph $D$, its \emph{cycle graph} $C(D)$ is the graph whose vertex-set consists of all directed cycles of $D$, with two of them adjacent if and only if they intersect in at least one vertex.
\end{definition}

We next prove two lemmas relating properties of this auxiliary graph to the original digraph.

\begin{lemma}\label{lem:diameter-in-cycle-graph}
For any connected vertex transitive digraph $D$ with directed diameter $d$ and circumference $\ell$ its cycle graph is connected with diameter at least $\frac{d}{\ell}-1.$
\end{lemma}
\begin{proof}
Suppose $C_1, C_2 \in C(D)$ and fix $v_1\in C_1, v_2 \in C_2$. Since $D$ is strongly connected, there is a directed path with vertex sequence $v_1=u_0,u_1,\ldots,u_\ell=v_1$ from $v_1$ to $v_2$ in $D$. Again by strong connectivity, for each $1\le i \le\ell$ there exists a directed cycle $C_i'$ in $D$ through the arc $(u_{i-1},u_i)$. Now, the sequence $C_1,C_1',\ldots,C_\ell',C_2$ yields a walk in $C(D)$ from $C_1$ to $C_2$. This proves that $C(D)$ is connected. 

For the diameter bound, take $v_1,v_2 \in V(D)$ such that the shortest directed path from $v_1$ to $v_2$ in $D$ has length $d$, and take arbitrary directed cycles $C_1$ containing $v_1$ and $C_2$ containing $v_2$. Let $C_0',C_1',\ldots, C_{t}'$ be a shortest path with $C'_0=C_1$ and $C'_t=C_2$ in $C(D)$. Starting at $v_1$ and following $C_0'$ until we reach the first vertex of $C_1'$ and then following $C_1'$ until we reach the first vertex of $C_2'$ and so on, we obtain a directed walk from $v_1$ to $v_2$ of length at most $(t+1) \ell$ in $D$. Since $v_1,v_2$ are at directed distance $d$ in $D$, this implies $t \ge \frac{d}{\ell}-1$, as claimed.
\end{proof}

\begin{lemma}\label{lem:induced-to-directed-cycles}
For any digraph $D$, if $C(D)$ contains an induced cycle of length $\ell\ge 4$, then $D$ contains a directed cycle of length at least $\ell$.
\end{lemma}
\begin{proof}

Let $C_1,\ldots, C_\ell \in C(D)$ be an induced cycle. Note that since $\ell \ge 4$, $C_2$ and $C_{\ell}$ are not adjacent in $C(D)$ and so $C_1 \cap C_\ell$ and $C_1 \cap C_2$ are disjoint. Pick vertices $v_1 \in C_1 \cap C_{\ell}$ and $v_2 \in C_1 \cap C_2$ that are joined by a directed path from $v_1$ to $v_2$ in $C_1$ with no other vertices of $C_\ell$ or $C_2$ on it. We now repeat the following for $2 \le i \le \ell-1$. Given $v_i \in C_{i} \cap C_{i-1}$ we follow $C_i$ until we first reach a vertex of $C_{i+1}$ which we set to be $v_{i+1}$. The key property we maintain by picking the first such vertex is that the path from $v_{i}$ to $v_{i+1}$ on $C_{i}$ does not contain other vertices of $C_{i+1}$, so it will in particular be disjoint from the path from $v_{i+1}$ to $v_{i+2}$ on $C_{i+1}$ (except for $v_{i+1}$). Finally, once we reach $v_{\ell}$ we simply close the cycle by walking along $C_{\ell}$ until $v_1$, which we can, since the path from $v_1$ to $v_2$ on $C_1$ does not contain any vertices of $C_{\ell}$ except for $v_1$. Hence, $v_1v_2\ldots v_{\ell}v_1$ is a directed cycle of length at least $\ell$ in $D$. 
\end{proof}

Vertex transitivity of a digraph $D$ might not be inherited in full by its cycle graph $C(D)$. It turns out that $C(D)$ does, in fact, inherit the following slightly weaker notion of transitivity. 

\begin{definition}\label{defn:near-transitive}
A graph $G$ is \emph{nearly transitive} if for any two vertices $v,u \in V(G)$, there exists an automorphism of $G$ mapping $v$ to $u$ or to a neighbor of $u$.
\end{definition}

Note that, in particular, a vertex transitive graph is nearly transitive. 
Near transitivity serves as a useful approximation of vertex transitivity for our purposes, in the sense that while in a vertex transitive graph one can map any vertex to any other vertex, in a nearly transitive graph one can map any vertex very close to any other vertex (namely, to a vertex at distance at most one). The following easy lemma shows that vertex transitivity of a digraph $D$ implies near transitivity of its cycle graph.

\begin{lemma}\label{lem:near-transitivyity-of-cycle-graphs}
	For any connected vertex transitive digraph $D$, its cycle graph $C(D)$ is nearly transitive.
\end{lemma}
\begin{proof}
	Note first that any automorphism $\phi:V(D) \to V(D)$ gives rise to a natural automorphism $\phi'$ of $C(D)$ which maps any $C \in V(C(D))$ to its image under $\phi$. 
	Now let $C_1,C_2 \in V(C(D))$ and pick $v \in C_1, u \in C_2$. Pick an automorphism $\phi$ of $D$ which maps $v$ to $u$. Now, $\phi'(C_1)$ is a cycle of $D$ passing through $u$, so in particular it intersects $C_2$ and is hence either equal to $C_2$ or adjacent to $C_2$ in $C(D)$.
\end{proof}

The following is the key lemma of this section, which we find interesting in its own right. We note that questions of similar flavor have been considered before, see e.g.\ \cite{closing-cycle,milanivc2025closing} (although with a very different focus). It guarantees the existence of a long induced cycle in a nearly transitive graph.

\begin{lemma}\label{lem:long-cycle-in-cycle-graphs}
	In any connected, nearly transitive graph $G$ with diameter $d\ge 20$, there is an induced cycle of length at least $d-17$.
\end{lemma}
\begin{proof}
	Suppose towards a contradiction that every induced cycle in $G$ has length less than $\frac{d}{4}$. 

	Let us fix a shortest path $S$ between two vertices $v,u$ at distance $d$ in $G$. Note that $S$ is a geodesic path and is hence induced. 
	Let $m$ be a vertex of $S$ at distance at least $\frac{d-1}{2}$ from both $u$ and $v$. 
	Let $L$ be the subpath of $S$ between $v$ and $m$ and let $R$ be the subpath between $m$ and $u$.
	
	Let $P$ be an induced path in $G$ of maximum length subject to the condition that the subpath $Q$ of $P$ on its last $\lceil \frac{d-5}{2}\rceil$ vertices is a geodesic path. Note that $S$ satisfies this condition and so such a path exists and has length at least $d$. Let $w$ be the common endpoint of $Q$ and $P$, let $x$ be the other endpoint of $P$ and let $y$ be the other endpoint of $Q$.
	
	Let $S'$ be the image of $S$ under an automorphism $\phi$ mapping $m$ to $w$ or a neighbor of $w$. Let 
    $L':=\phi(L), R':=\phi(R),$ $v':=\phi(v)$, $u':=\phi(u)$ and $w':=\phi(m)$.

    We first claim that either $L'$ or $R'$ intersects $Q \cup N(Q)$ only in vertices at distance at most $3$ from $w'$. To see this, let $a'$ and $b'$ be vertices of $L' \cap (Q \cup N(Q))$ and $R' \cap (Q \cup N(Q))$ respectively. Let $a$ and $b$ be vertices of $Q$ at distance at most one from $a'$ and $b'$, respectively. See \Cref{fig:translation} for an illustration of the setup so far.

    \begin{figure}[h]
    \centering
\includegraphics[scale=0.8]{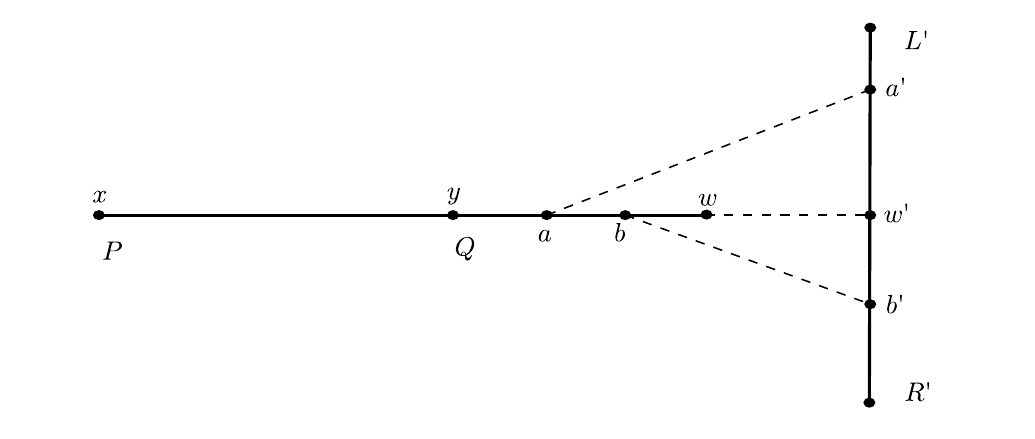}
    \caption{Illustration of the setup in the proof of \Cref{lem:long-cycle-in-cycle-graphs}. Dashed lines depict pairs of vertices at distance up to one. $L'\cup R'=S'$ is a geodesic path, and so is $Q$. This implies that $d(w,a)$ is very close to $d(w,a')$ and $d(w,b)$ to $d(w,b')$. The argument now relies on the fact that $d(a',b')$ is roughly the sum of these two distances since $a',b'$ are both vertices of the geodesic path $S'$, but also by going through the $a$-$b$ segment, we can find a path of distance roughly the difference between these two distances. The conclusion is that one of these distances needs to be very small.}
    \label{fig:translation}
\end{figure}

    Note that since $d(w,w'),d(a,a'),d(b,b') \le 1$, we get via triangle inequality that 
    \begin{align}\label{eq:1}
        |d(a',w')-d(a,w)|&\le d(a',a)+d(w,w') \le 2 \quad \text{ and}\\
        \label{eq:2}
        |d(b',w')-d(b,w)|&\le d(b',b)+d(w,w') \hspace{0.05cm}\le 2.
    \end{align}
    
    Since $S'$ is geodesic (being an image of a geodesic path under an automorphism), so are all of its subpaths. Combining this with several applications of the triangle inequality, we get
    \begin{equation}\label{eq:3}
        d(a', w')+d(b',w')=d(a',b') \le d(a,b)+2 = |d(a,w)-d(b,w)|+2.
    \end{equation} 
    Combining \eqref{eq:1}, \eqref{eq:2} and \eqref{eq:3} we obtain 
	$$d(a',w')+d(b',w') \le |d(a',w')-d(b',w')|+6 \implies \min\{d(a',w'),d(b',w')\}\le 3,$$
	 as claimed.
	
	Suppose without loss of generality
    that $L'$ intersects $Q \cup N(Q)$ only in vertices at distance at most $3$ from $w'$. Pick a vertex $c$ in this intersection farthest from $w'$. Let $z$ be a vertex at distance at most one from $c$ on $Q$, farthest from $w$. By the triangle inequality, we know that $d(z,w) \le d(c,w') +2 \le 5$. Now, let $Q'$ be the subpath of $Q$ from $y$ to $z$, and let $L''$ be the subpath of $L'$ from $c$ to $v'$. Note that $V(Q') \cup V(L'')$ induces a path by construction (since $c$ and $z$ are equal or adjacent). Let $P'$ denote the subpath of $P$ from $x$ to $y$.
    We now claim that there must be a pair of vertices $s\in V(P')$ and $t\in V(L'')$ with $d(s,t)\le 1$. Indeed, suppose for a contradiction such a pair does not exist. Then $P'$ and $L''$ are vertex-disjoint and have no connecting edges. In particular, appending the path induced by $V(Q')\cup V(L'')$ to $P'$, we obtain an induced path whose last $\lceil \frac{d-5}{2}\rceil$ vertices induce a subpath of $L'$ and thus a geodesic path. Moreover, the total length of the new path is at least $|E(P)|+\frac{d-1}{2}-8+1 > |E(P)|$, and this contradicts our choice of $P$. Hence, there indeed need to exist vertices $s\in V(P'), t \in V(L')$ with $s=t$ or $st\in E(G)$. Pick such $s$ and $t$ for which $d_{P'}(s,y)+d_{L''}(t,c)$ is minimized. Let $C$ be the cycle formed by the union of the segment of $P'$ from $s$ to $y$, the segment of $L''$ from $t$ to $c$, the edge $st$ if $s\neq t$, as well as the path in $G$ induced by $V(Q')\cup V(L'')$. It is not hard to see that by our choice of $c,z,s,t$ respectively $C$ forms an induced cycle in $G$. Since $C$ by definition contains $Q'$ as a subpath and since $Q'$ is geodesic, we find that $|E(C)|\ge 2|E(Q')|\geq 2(\frac{d-7}{2}-5)=d-17$, as desired. 
\end{proof}

We are now ready to formally put together the proof of \Cref{thm:cycles-long}.

\begin{proof}[Proof of \Cref{thm:cycles-long}]
    If our digraph has diameter $d \le n^{2/3}$, then by \Cref{lem:transitive-implies-expander,lem:long-cycles-in-expanders} we can find a directed cycle of length $\Omega(n^{1/3})$. 

    Suppose now $d \ge n^{2/3}$. By \Cref{lem:diameter-in-cycle-graph} we either have a cycle in $D$ of length at least $\Omega(n^{1/3})$ or $C(D)$ has diameter at least $O(n^{1/3})$. 

    By \Cref{lem:long-cycle-in-cycle-graphs,lem:near-transitivyity-of-cycle-graphs} we can now find an induced cycle in $C(D)$ of length $\Omega(n^{1/3})$ which via \Cref{lem:induced-to-directed-cycles} gives a directed cycle in $D$ of length $\Omega(n^{1/3})$, completing the proof.
\end{proof}

\section{Perimeter gap in vertex transitive digraphs}\label{sec:gap}

It is well-known (see, e.g.\ Trotter and Erd\H{o}s~\cite{TrotterErdos}) that there exist arbitrarily large connected vertex transitive digraphs which are not Hamiltonian. In fact, an easy counterexample is also given by a natural ``toroidal'' version of our construction depicted in \Cref{fig:disjoint}, where the bidirected edges are replaced by edges wrapping around, provided it consists of $8n+4$ vertices,  for any $n \ge 1$. In this section, we prove Theorem~\ref{thm:cycles-short}, showing that for infinitely many $n$ there are connected vertex transitive digraphs of order $n$ with perimeter gap as large as $(1-o(1))\ln n$. 

To do so, we follow~\cite{TrotterErdos} and consider Cartesian products of directed cycles. Recall that given two digraphs $D_1, D_2$, the \emph{Cartesian product} $D_1\car D_2$ is the digraph with vertex-set $V(D_1)\times V(D_2)$ in which there is an arc from a vertex $(u_1,u_2)$ to another vertex $(v_1,v_2)$ if and only if $u_1=v_1$ and $(u_2,v_2)\in A(D_2)$ or $u_2=v_2$ and $(u_1, v_1)\in A(D_2)$. It is not hard to check that the Cartesian product of two vertex transitive digraphs is still vertex transitive. Trotter and Erd\H{o}s~\cite{TrotterErdos} obtained a characterization of when the cartesian product $\vec{C}_{n_1}\car \vec{C}_{n_2}$ of two directed cycles is Hamiltonian (Rankin \cite{rankin} actually obtained a less direct characterization which applies in far more generality and, in particular, easily implies this result, see \cite[Section 5]{path-non-hamiltonicity}). We will only need the necessity part, which we state below, and include a short proof for the sake of completeness. 

\begin{theorem}[cf.~Theorem~1 in~\cite{TrotterErdos}]\label{thm:trottererdos}
Let $n_1, n_2\ge2$ be integers, and let $d:=\mathrm{gcd}(n_1,n_2)$. If $\vec{C}_{n_1}\car \vec{C}_{n_2}$ admits a directed Hamiltonian cycle, then $d \ge 2$ and there exist positive integers $d_1, d_2$ such that $d_1+d_2=d$ and $\mathrm{gcd}(n_i,d_i)=1$ for $i=1,2$. 
\end{theorem}

\begin{proof}
    It will be convenient to view $\vec{C}_{n_1}\car \vec{C}_{n_2}$ as a Cayley-digraph $D$ of the group $\mathbb{Z}_{n_1}\times \mathbb{Z}_{n_2}$ with generator set $S=\{(1,0),(0,1)\}$. Suppose that indeed we have a directed Hamilton cycle $C$. This gives a natural partition of $V(D)$ into a set $V^{\rightarrow}$ of vertices from which $C$ continues along a $(1,0)$ edge and $V^{\uparrow}$ as those from which it continues along a $(0,1)$ edge. Note that $|V^{\rightarrow}|+|V^{\uparrow}|=n_1n_2$ and that $V^{\rightarrow},V^{\uparrow} \neq \emptyset$ since $n_1,n_2 \ge 2$ imply that following only one type of edges does not create a Hamilton cycle.

    Observe next that for any $v \in V^{\uparrow}$ we also have\footnote{All vector operations are done in $\mathbb{Z}_{n_1}\times \mathbb{Z}_{n_2}$.} $v+(1,-1) \in V^{\uparrow}$. Indeed, since $v \in V^{\uparrow}$, we know that $v+(1,0)$ could not be preceded by $v$ in $C$, leaving $v+(1,-1)$ as the only option. This implies that $d=\gcd(n_1,n_2) \ge 2$, as otherwise $(1,-1)$ generates $\mathbb{Z}_{n_1}\times \mathbb{Z}_{n_2}$ by the Chinese Remainder Theorem and we would have either $V^{\uparrow}$ or $V^{\rightarrow}$ being the whole set (making the other empty, which is impossible).

    Similarly, if $v \in V^{\uparrow}$ we also have $v+(d,0) \in V^{\uparrow}$ since $(d,0) \in \langle (1, -1) \rangle$. Indeed, the system $a \equiv 1 \bmod {\frac{n_1}{d}}$ and $a \equiv 0 \bmod {\frac{n_2}{d}}$ has a solution since $\gcd(\frac{n_1}{d},\frac{n_2}{d})=1$. If $a$ is this solution we have $ad(1,-1)=(d,0)$, as claimed. Furthermore, since $(d-j,j)=(d,0)-j(1,-1)$ we in fact conclude $v \in V^{\uparrow} \implies v+(d-j,j) \in V^{\uparrow}$, for any $j$. If we write $C=v_1v_2\ldots v_n$ this implies that $v_i \in V^{\uparrow} \Leftrightarrow v_{i+d}\in V^{\uparrow}$ since $v_{i+d}=v_i+j(1,0)+(d-j)(0,1)= v_i+(j,d-j)$ for some $j$. 
    
    Let now $v_{d+1}=v_1+(d_1,d_2)$ where $d_1+d_2=d$. Note that this implies that exactly $d_1$ vertices among $v_1,\ldots, v_d$ are in $V^{\rightarrow}$ and exactly $d_2$ in $V^{\uparrow}$. By our observations above, we further know that the same is true about any set of vertices $v_{id+1},\ldots, v_{(i+1)d}$. So, in particular, we can conclude that $v_{kd+1}=v_1+k(d_1,d_2)$.
    This implies that the order $o$ of $(d_1,d_2)$ in $\mathbb{Z}_{n_1}\times \mathbb{Z}_{n_2}$ equals $\frac{n_1n_2}{d}$.

    On the other hand, if we write $o_i$ for the order of $d_i$ in $\mathbb{Z}_{n_i}$ we have that $o=\mathrm{lcm}(o_1,o_2)$.
    If there existed a prime $p$ dividing both $n_i$ and $d_i$, then $o_i \mid \frac{n_i}{p}$, so assuming also $p \nmid d$ we get $o=\mathrm{lcm}(o_1,o_2) \le \mathrm{lcm}(\frac{n_i}{p},n_{3-i})\le \frac{n_1n_2/p}{\gcd(n_i/p,n_{3-i})}\le \frac{n_1n_2}{pd}<\frac{n_1n_2}{d},$ a contradiction. Similarly, if $p \mid d$, we can conclude also that $p \mid d-d_i=d_{3-i}$, so $p$ divides both $n_1,n_2$ and we have  $o=\mathrm{lcm}(o_1,o_2) \le \mathrm{lcm}(n_1/p,n_{2}/p)\le \frac{n_1n_2/p^2}{\gcd(n_1/p,n_2/p)}\le \frac{n_1n_2/p^2}{d/p}<\frac{n_1n_2}{d}.$ So $\gcd(n_i,d_i)=1$ for both $i=1,2$, as claimed.
\end{proof}

Using \Cref{thm:trottererdos}, one can obtain the following lower bound on the perimeter gap of $\vec{C}_{n_1}\car \vec{C}_{n_2}$.

\begin{lemma}\label{lem:pergap}
Let $n_1, n_2\in \mathbb{N}$, let $d:=\mathrm{gcd}(n_1,n_2)$, and suppose that for all positive integers $d_1, d_2$ such that $d_1+d_2=d$ we have $\mathrm{gcd}(n_1,d_1)\ge 2$ or $\mathrm{gcd}(n_1,d_1)\ge 2$. Then, the perimeter gap of $\vec{C}_{n_1}\car \vec{C}_{n_2}$ is at least $d$.
\end{lemma}
\begin{proof}
    Let $C$ be a maximum length directed cycle in $D:=\vec{C}_{n_1}\car \vec{C}_{n_2}$. As before, we view $D$ as a Cayley-digraph of the group $\mathbb{Z}_{n_1}\times \mathbb{Z}_{n_2}$ with generator set $S=\{(1,0),(0,1)\}$ and we classify the arcs on $C$ into two types: Those arcs which correspond to the element $(1,0)$ of $S$ and those which correspond to the element $(0,1)$ of $S$. Let $\ell_1, \ell_2$ denote the number of arcs on $C$ of the first and second type, respectively. We can then see that, since $C$ is a directed cycle, we must have $\ell_1(1,0)+\ell_2(0,1)=(0,0)$, i.e., $\ell_1$ is divisible by $n_1$ and $\ell_2$ is divisible by $n_2$. In particular, both $\ell_1$ and $\ell_2$ must be divisible by $d$, and hence the length $|C|=\ell_1+\ell_2$ of $C$ must also be divisible by $d$. Since also $n=n_1n_2$ is divisible by $d$, it follows that either $|C|=n$ or $|C|\le n-d$. However, the first case is impossible by our assumption in the lemma, which together with Theorem~\ref{thm:trottererdos} rules out the existence of a directed Hamiltonian cycle. Hence, we indeed must have $|C|\le n-d$, proving that the perimeter gap is at least $d$. 
\end{proof}

Let us call a number $d\in \mathbb{N}$ \emph{prime partitionable}~\cite{TrotterErdos} (alternatively, these numbers are also referred to as Erd\H{o}s-Woods numbers) if there exist $n_1, n_2\in \mathbb{N}$ such that $d=\mathrm{gcd}(n_1,n_2)$ and for every pair of positive numbers $d_1, d_2$ with $d_1+d_2=d$ we have $\mathrm{gcd}(n_i,d_i)>1$ for some $i\in \{1,2\}$. In this case, we call $(n_1,n_2)$ a \emph{witness} for $d$ being prime partitionable. According to Lemma~\ref{lem:pergap}, for every prime partitionable number $d\in \mathbb{N}$ with witness $(n_1,n_2)$ there exists a connected vertex transitive digraph on $n_1n_2$ vertices with perimeter gap at least $d$. Hence, Alspach's question has a positive answer provided we can show that there exist infinitely many prime partitionable numbers. Curiously, Trotter and Erd\H{o}s~\cite{TrotterErdos} claimed this result in their paper from 1979. However, their proof was flawed, as we want to explain now: Their proof was starting from the assumption (claimed as a fact in the paper by Trotter and Erd\H{o}s) that there exist infinitely many pairs of prime numbers $p_1, p_2$ such that $p_2=2p_1+1$. However, a prime number $p_1$ such that $2p_1+1$ is also a prime is called a \emph{Sophie-Germain prime}, and until today it remains a widely open problem whether infinitely many such numbers exist. 
In the following, we present a modified argument avoiding this assumption. We shall also need an explicit quantitative bound on the numbers $n_1, n_2$ with $\mathrm{gcd}(n_1,n_2)=d$ certifying that a certain number $d$ is prime partitionable in order to obtain our \Cref{thm:cycles-short}. In the proof, we make use of the following result due to Motohashi, which is one of many results related to estimates of the famous \emph{Linnik constant} in number theory~\cite{linnik1,linnik2}.

\begin{theorem}[Motohashi~\cite{motohashi70}]\label{thm:motohashi}
There exists an absolute constant $\vartheta<1.64$ such that, for infinitely many primes $p$, there exists another prime $q$ such that $q\equiv 1~(\mathrm{mod}\text{ }p)$ and $q<p^\vartheta$. 
\end{theorem}

We shall deduce the following consequence of Motohashi's theorem.

\begin{lemma}\label{lem:primepartitionable}
There are infinitely many numbers $d\in \mathbb{N}$ such that $d$ is prime-partitionable, and this can be witnessed by numbers $(n_1,n_2)$ such that $n_1n_2\le e^{d+o(d)}$. 
\end{lemma}

\begin{proof}
Let $\vartheta<1.64$ be the constant from Theorem~\ref{thm:motohashi}. Let us pick arbitrarily one of the infinitely many prime numbers $p$ satisfying the statement of Theorem~\ref{thm:motohashi}. We will define $d, n_1, n_2$ depending on $p$ satisfying the desired properties as follows. First, we pick some prime number $q$ such that $q\equiv 1 \text{ }(\mathrm{mod}\text{ }p)$ and such that $q\le p^\vartheta$ (which exists by Theorem~\ref{thm:motohashi} and our choice of $p$). 

Now, we define $d:=p+q$, $n_1:=dpq$ and $n_2:=d\cdot \prod_{z\in \mathbb{P}\setminus \{p,q\}, z<d}z$. 

We claim that $d$ is prime partitionable with witness $(n_1,n_2)$. By construction, $\mathrm{gcd}(n_1,n_2)=d$. Next, consider any positive numbers $d_1, d_2$ such that $d_1+d_2=d$, and let us show $\mathrm{gcd}(n_i,d_i)>1$ for some $i\in \{1,2\}$. Towards a contradiction, suppose that $n_1, d_1$ and $n_2, d_2$ are coprime. Then, by definition of $n_2$, we immediately see that $d_2$ cannot have any prime factor distinct from $p$ or $q$. Furthermore, $q^2>pq>p^2>p+p^\vartheta>p+q=d$ by our choice of $p$ and $q$ (without loss of generality we may assume that $p$ is sufficiently large). Since $d_2<d$, it follows that $d_2\in \{1,p,q\}$. However, this implies that $d_1\in \{p+q-1,p,q\}$. Since neither $p$ nor $q$ is coprime with $n_1$ by definition, we conclude that $d_1=p+q-1$. But since $q\equiv 1 \text{ }(\mathrm{mod}\text{ }p)$, it then follows that $d_1$ is divisible by $p$, again a contradiction, since $n_1$ is also divisible by $p$. Having obtained a contradiction in all possible cases, we conclude that we indeed must have $\mathrm{gcd}(n_i,d_i)>1$ for some $i\in \{1,2\}$. Hence, we have shown that $d$ is prime partitionable with witness $(n_1,n_2)$. Since we create infinitely many distinct numbers $d$ when making infinitely many distinct choices for $p$ in this construction, this concludes the proof of the first part of the statement of the lemma. It remains then to verify that $n_1, n_2$ as defined above satisfy $n_1n_2\le e^{d+o(d)}$. This, however, is easy to check: By the prime number theorem there are at most $(1+o(1))\frac{d}{\ln(d)}$ prime numbers smaller than $d$, which implies:
$$n_1n_2=d^2\prod_{z\in \mathbb{P}, z<d}z\le d^2\cdot d^{(1+o(1))d/\ln d}=e^{d+o(d)},$$ as desired.
\end{proof}

Combining Lemma~\ref{lem:pergap} and Lemma~\ref{lem:primepartitionable} now directly implies Theorem~\ref{thm:cycles-short}.
\begin{proof}[Proof of Theorem~\ref{thm:cycles-short}]
By Lemma~\ref{lem:primepartitionable} there are infinitely many numbers $d$ such that there exist $(n_1,n_2)$ with $n:=n_1n_2\le e^{d+o(d)}$ witnessing that $d$ is prime partitionable. By Lemma~\ref{lem:pergap} we then have that the perimeter gap of the vertex transitive digraph $\vec{C}_1\car\vec{C}_2$ on $n$ vertices is at least $d\ge (1-o(1))\ln n$. Since we certainly generate infinitely many distinct values of $n$ this way, this proves the assertion of the theorem.
\end{proof}

\section{Concluding remarks}\label{sec:conc-remarks}

In this paper, we proved the first growing bounds on the possible perimeter gap of vertex-transitive digraphs, answering a question of Alspach from 1981, as well as on the circumference of such graphs. The most immediate question is to determine the best possible bounds. However, this seems far out of reach, given that we do not even know the answer in the undirected case (which is a special case). We raise two natural intermediate targets below which we believe might be more attainable. 

\begin{conj}
    There exists an $\eps >0$ and infinitely many values of $n$ for which there exists a connected vertex transitive digraph of order $n$, whose perimeter gap is at least $\eps n$.
\end{conj}

In fact, it would already be interesting to improve our logarithmic bound from \Cref{thm:cycles-short} to a polynomial one.

On the other hand, we prove that such digraphs always have cycles of length at least $\Omega(n^{1/3})$. Given that this is slightly weaker than the best known bounds in the undirected setting, it would be interesting to at least asymptotically match the best known bounds or even to show that, in general, the answers in the directed and undirected cases are asymptotically the same.

\begin{conj}
    Let $c(n)$ denote the minimum circumference of a connected vertex transitive graph on $n$ vertices. Let $d(n)$ denote the minimum circumference of a connected vertex transitive digraph on $n$ vertices. Then $d(n)= \Theta(c(n))$.
\end{conj}


As mentioned in the introduction and illustrated by \Cref{fig:disjoint}, a major obstacle to translating results from graphs to digraphs is that, in contrast to $2$-connected graphs, strongly $2$-connected digraphs may contain vertex-disjoint longest cycles.
To overcome this obstacle, it makes sense to consider conditions on digraphs which might guarantee that every pair of longest cycles intersect.
We wonder if vertex transitivity could be such a condition.


\begin{qn}
    Is it true that in any connected vertex transitive digraph, any two longest directed cycles intersect?
\end{qn}




Another very natural question to ask is: What happens in the infinite case? Here, if one removes the bidirected edges in \Cref{fig:disjoint} and extends it on both sides to infinity, one obtains an example of an infinite strongly $2$-connected vertex-transitive digraph that has no directed cycles of length greater than four. However, this example is not strongly $3$-connected, and it is natural to ask if one can make such a construction that is strongly $k$-connected for any constant $k$.
One can show that strong connectivity together with high minimum degree is insufficient to guarantee long cycles in this setting by considering the digraph obtained from a $k$-regular tree by replacing each vertex by an independent set of size $2$ and each edge by a directed $C_4$.

Another interesting question is whether the answer is the same if we look for a long path instead of a cycle in a connected vertex transitive digraph. In the undirected setting, this is the case, at least up to a constant factor, thanks to a classical result in this direction (see, e.g.\ \cite{path-vs-cycle}).

\begin{qn}
    Does there exist $C>0$ such that in any vertex transitive digraph, the length of a longest path is by at most a factor of $C$ larger than the length of a longest directed cycle?
\end{qn}

\Cref{thm:cycles-short} shows that the two lengths can differ by at least an additive logarithmic term. Indeed, our construction uses a Cayley digraph over an abelian group, and it is a classical result of Holszty\'nski and Strube \cite{abelian-path} that every Cayley digraph over an abelian group contains a directed Hamiltonian path. 

It is known for a long time that there exist Cayley digraphs that do not contain a Hamilton path. The oldest claimed instance of such a construction appears as an exercise in an early textbook on algorithms by Nijenhuis and Wilf \cite{nijenhuis-wilf}. However, this was later shown to be false in \cite{disproof}. Another construction, attributed to Milnor, was stated in \cite{nathanson} without proof, and was extended and improved upon by Morris in \cite{path-non-hamiltonicity}. It would be interesting to prove a variant of \Cref{thm:cycles-short} where the longest directed path has length at most $n-\omega(1)$.  

\bigskip

\textbf{Acknowledgments.} This work was initiated at the 2025 Barbados Graph Theory workshop in Holetown, Barbados. The authors thank the organizers for creating a stimulating working atmosphere. The fourth author would like to thank Anders Martinsson for interesting discussions related to the subject of this paper.

\providecommand{\MR}[1]{}
\providecommand{\MRhref}[2]{%
\href{http://www.ams.org/mathscinet-getitem?mr=#1}{#2}}

  \bibliographystyle{amsplain_initials_nobysame}
  \bibliography{refs}

\end{document}